\documentclass[12pt]{article}
\usepackage[margin=1.25in]{geometry}
\usepackage{pgfplots}
\usepackage{subfig}
\usepackage{amsmath}
\usepackage{amssymb}
\usepackage{algorithmic}
\usepackage{algorithm}
\usepackage{mathtools}
\usepackage{color}
\usepackage[normalem]{ulem}
\usepackage{setspace}
\usepackage{tikz}
\usetikzlibrary{arrows,
	arrows.meta,
	bending,calc,patterns,angles,quotes}
\usepackage{float}
\usepackage{hyperref}

\makeatletter
\providecommand*{\cupdot}{%
	\mathbin{%
		\mathpalette\@cupdot{}%
	}%
}
\newcommand*{\@cupdot}[2]{%
	\ooalign{%
		$\m@th#1\cup$\cr
		\hidewidth$\m@th#1\cdot$\hidewidth
	}%
}
\makeatother

\newtheorem{example}{Example}

\newtheorem{theorem}{Theorem}

\newtheorem{remark}{Remark}
\newtheorem{definition}{Definition}

\def\endpf{{\ \hfill\hbox{\vrule width1.0ex height1.0ex}\parfillskip 0pt
	}}

\newcounter{figurecounter}
\begin{document}
\setlength{\baselineskip}{20pt}

\title{When does admission control reduce congestion?
\\A stochastic ordering approach}

\author{Royi Jacobovic\footnote{School of Mathematical Sciences, Tel-Aviv University, Tel-Aviv, Israel, 6997800, E-mail: royijacobo@tauex.tau.ac.il.} \footnote{Jacobovic acknowledges the support of the Israel Science Foundation, Grant \#3739/24.} \and Daniel Stiebel \footnote{School of Mathematical Sciences, Tel-Aviv University, Tel-Aviv, Israel, 6997800, E-mail: stiebel1@mail.tau.ac.il}}
\maketitle

\begin{abstract}
\setlength{\baselineskip}{20pt}
Admission control is a widely used mechanism for regulating congestion in stochastic service systems, where restricting arrivals is expected to reduce workload and improve performance. 

In many settings, however, customers make decisions based on the observed system state, creating a feedback loop between control actions and future arrivals. This raises a fundamental question: does admission control necessarily reduce congestion under state-dependent behavior?

We address this question using stochastic ordering of workload processes. We show that admission control may fail to reduce congestion due to endogenous feedback effects, and construct explicit counterexamples illustrating this phenomenon.

We then identify a general condition under which admission control is effective over the initial busy period. In particular, independence between customers’ types and service requirements eliminates the adverse feedback mechanism and restores stochastic dominance.

These results highlight fundamental limitations of admission control and the importance of accounting for behavioral responses in the design of stochastic service systems.

\end{abstract}


\smallskip


\section{Introduction}

Service systems are a fundamental component of modern infrastructure, encompassing applications such as communication networks, cloud computing platforms, call centers, and transportation systems. A central challenge in these systems is congestion, which arises when incoming demand temporarily exceeds service capacity. Managing congestion is therefore a key objective in the design and control of such systems.

Admission control mechanisms are widely implemented in practice as a means of regulating system performance under uncertainty. In many service systems, the decision to accept or reject incoming customers plays a critical role in maintaining quality of service, controlling delays, and preventing overload. Examples include cloud computing platforms that throttle incoming jobs to mitigate latency, call centers that restrict access during peak demand, and healthcare systems that rely on triage policies under congestion. 

A common feature of these systems is that users adapt their behavior to observable system conditions, such as delay or workload. As a result, admission control is not merely a passive filtering mechanism, but an active component in a dynamic system with feedback, where control actions may influence future arrivals and system evolution.

In this context, \emph{admission control} provides a natural mechanism for regulating congestion by selectively accepting or rejecting incoming customers based on the current state of the system. At an intuitive level, admission control is expected to reduce congestion: by limiting the inflow of customers, the system workload should decrease, leading to improved performance.

Yet, this intuition relies on a critical implicit assumption: that customer behavior is unaffected by the control policy. In many stochastic service systems, however, customers react to the observed workload, creating a feedback loop between control actions and future arrivals. In such settings, it is no longer clear whether admission control improves or degrades system performance.

This observation raises a fundamental question:
\begin{quote}\centering
\emph{Does admission control necessarily reduce congestion when customers respond to the system state?}
\end{quote}

In this work, we address this question through the lens of \emph{stochastic ordering of workload processes}, which provides a natural way to compare entire system trajectories rather than pointwise performance measures. Specifically, we say that an admission control policy is effective if the resulting workload process is stochastically dominated by that of the corresponding uncontrolled system.

Our analysis reveals that the targeted benefit of admission control does not hold in full generality. In particular, we identify a key mechanism driving this failure: although admission control locally reduces workload by rejecting customers, it may also modify the system state in a way that influences future customer decisions. This \emph{feedback effect} can lead to higher congestion levels in the controlled system. Intuitively, by reducing congestion, admission control may make the system more attractive to certain customers, thereby increasing future arrivals and potentially offsetting its intended effect.

To formalize this phenomenon, we construct explicit counterexamples highlighting situations in which admission control does not reduce congestion in the sense of stochastic dominance. These examples reveal two distinct mechanisms: one driven by the admission of customers with large service requirements, and another arising from the cumulative effect of multiple smaller arrivals. Together, these mechanisms show that even natural admission control policies may backfire in the presence of state-dependent customer behavior, as control actions reshape future arrival patterns and create a feedback loop between system dynamics and user decisions.

This naturally raises the question of whether structural conditions can prevent such adverse feedback effects. To complement the negative results, we identify a general structural condition under which admission control regains its effectiveness. In particular, we show that when customers’ types and service requirements are independent, the harmful feedback mechanism disappears, and admission control induces a stochastic reduction of workload over the initial busy period. 

Taken together, these results provide a key insight: \begin{quote}\centering\emph{The effectiveness of admission control critically depends on the interaction between control actions and state-dependent user behavior, which in turn shapes future arrivals.}
\end{quote}
This highlights the need to move beyond purely exogenous models of arrivals when analyzing controlled stochastic service systems. 

Beyond its theoretical interest, our work has broader implications for the design of stochastic service systems.  It suggests  that neglecting behavioral feedback in the design of control policies may lead to unintended and potentially adverse system-level outcomes. This perspective is particularly relevant in modern applications where users adapt their behavior to observable system conditions, such as delays, congestion, or service quality, as in cloud computing platforms, call centers, and online marketplaces.

\begin{figure}
\centering

\begin{minipage}{0.8\textwidth}
\centering
\begin{tikzpicture}[>=Stealth, scale=1]

\node[draw, rounded corners, align=center, minimum width=4cm] (state)
at (0,0) {System state};

\node[draw, rounded corners, align=center, minimum width=4.5cm] (control)
at (5,1.5) {Admission control\\(policy decision)};

\node[draw, rounded corners, align=center, minimum width=4.5cm] (behavior)
at (5,-1.5) {Customer behavior\\(balking decision)};

\node[draw, rounded corners, align=center, minimum width=5cm] (arrival)
at (0,-3) {Effective arrival process};

\draw[->, blue, thick] (state) -- (control);

\draw[->, green!60!black, thick] (control) -- (arrival);
\draw[->, green!60!black, thick] (behavior) -- (arrival);

\draw[->, red, thick] (arrival) -- (state);

\end{tikzpicture}

\medskip
{\small (a) Without behavioral feedback}
\end{minipage}

\vspace{0.8cm}

\begin{minipage}{0.8\textwidth}
\centering
\begin{tikzpicture}[>=Stealth, scale=1]

\node[draw, rounded corners, align=center, minimum width=4cm] (state)
at (0,0) {System state\\(congestion level)};

\node[draw, rounded corners, align=center, minimum width=4.5cm] (control)
at (5,1.5) {Admission control\\(policy decision)};

\node[draw, rounded corners, align=center, minimum width=4.5cm] (behavior)
at (5,-1.5) {Customer behavior\\(balking decision)};

\node[draw, rounded corners, align=center, minimum width=5cm] (arrival)
at (0,-3) {Effective arrival process};

\draw[->, blue, thick] (state) -- (control);
\draw[->, blue, thick] (state) -- (behavior);

\draw[->, green!60!black, thick] (control) -- (arrival);
\draw[->, green!60!black, thick] (behavior) -- (arrival);

\draw[->, red, thick] (arrival) -- (state);

\end{tikzpicture}

\medskip
{\small (b) With behavioral feedback}
\end{minipage}

\caption{Comparison of feedback structures in service systems with admission control.
In both panels, the system state determines admission decisions, while the effective arrival process results from the combined effect of control actions and customer behavior, and drives the subsequent system evolution (red arrow).
In (a), customer behavior is independent of the system state, so admission control affects arrivals only through policy decisions.
In (b), customer behavior is state-dependent, introducing an additional feedback loop (blue arrow) from the system state to user decisions.
This behavioral feedback endogenizes the arrival process and may significantly alter the impact of admission control on congestion.}

\end{figure}

\subsection{Contributions}

The main contributions of this paper are as follows:
\begin{itemize}
    \item We introduce a trajectory-based notion of effectiveness for admission control using stochastic dominance of workload processes.
    \item We demonstrate that admission control may fail to reduce congestion due to endogenous feedback effects.
    \item We construct explicit counterexamples illustrating distinct mechanisms behind this failure.
    \item We establish a general condition under which admission control is effective over the initial busy period.
    \item We show how this result can be applied to derive stability conditions for state-dependent queueing systems with impatient customers.
\end{itemize}

\subsection{Related literature}

Admission control in service systems has been extensively studied in the queueing literature. A large body of work focuses on identifying optimal admission policies under various performance criteria. For example, Van Nunen and Puterman~\cite{Van Nunen1983} studied a GI/M/s queue with admission control in which customers may be accepted or rejected. Haviv and Puterman~\cite{Haviv1998} analyzed an M/M/1 queue with admission control and showed that gain-optimal stationary policies are of control-limit type, with at most two consecutive thresholds. Ata and Shneorson~\cite{Ata2006} examined a system with both admission and service-rate control and characterized optimal dynamic policies.

More recent contributions consider learning-based and revenue-maximizing formulations. For instance, Cohen, Subramanian, and Zhang~\cite{Cohen2024} studied an admission control problem in an M/M/1 queue with unknown parameters and proposed a learning-based dispatching algorithm that achieves bounded regret relative to the optimal policy. Carr and Duenyas~\cite{Carr2000} investigated admission control and sequencing decisions in production systems, while Maglaras and Zeevi~\cite{Maglaras2005} analyzed pricing and admission decisions in service systems with heterogeneous customers.

Another strand of the literature studies regulatory mechanisms in queues, where customers decide whether to join based on prices or other incentives. In this framework, a social planner sets an admission fee, and each arriving customer chooses whether to join or balk based on partial information about the system state. This setting typically leads to equilibrium joining behavior and can therefore be interpreted as an endogenous admission control mechanism. Seminal contributions include Naor~\cite{Naor1969} and Knudsen~\cite{Knudsen1972}, with comprehensive surveys provided by Hassin and Haviv~\cite{Hassin2003} and Hassin~\cite{Hassin2016}.

Our notion of effectiveness is formulated in terms of \emph{first-order stochastic dominance of workload functionals}, which connects our study to a broad literature on stochastic ordering in queueing systems. For example, Bodas and Jacobovic~\cite{Bodas2024} established stochastic dominance results for workload functionals over busy periods in time-varying queues with impatience. Bhattacharya and Ephremides~\cite{Bhattacharya1991} derived stochastic monotonicity results for multi-server queues with impatient customers. Jouini and Dallery~\cite{Jouini2007} obtained stochastic dominance relations for certain functionals in M/M/k/K+M systems.

Additional contributions on stochastic monotonicity in queueing networks include the works of Shanthikumar and Yao~\cite{Shanthikumar1986,Shanthikumar1987,Shanthikumar1989}, as well as Adan and Van der Wal~\cite{Adan1989} and Van Dijk and Van der Wal~\cite{Van Dijk1989}, who studied monotonicity properties in closed and tandem queueing networks. A related line of research originates in the seminal paper of Ross~\cite{Ross1978}, which investigated stochastic orderings in queues with nonstationary Poisson arrivals, followed by further developments by Heyman~\cite{Heyman1982} and Rolski~\cite{Rolski1981,Rolski1986}.

In contrast to these works, our focus is not on monotonicity with respect to model parameters, but rather on the structural effect of admission control on the entire workload trajectory.

\subsection{Structure of the paper}

The remainder of the paper is organized as follows.

In Section~2 we introduce the general probabilistic framework. We describe a single-server queueing system with general arrivals and a flexible admission control mechanism. To capture realistic behavioral features, each customer is characterized by a type and by partial information about the past evolution of the system, which together determine the customer's joining decision.

Section~3 introduces the notion of effectiveness based on stochastic dominance and provides an intuitive discussion explaining why admission control does not necessarily reduce congestion when joining decisions depend on the observed workload.

Section~4 formalizes this intuition by constructing two counterexamples showing that admission control policies may fail to be effective. Both examples arise in single-server queues with impatient customers and highlight different mechanisms through which admission control may increase congestion.

Section~5 introduces the notion of \emph{initial busy-period effectiveness} and presents the main positive result of the paper. We establish a general condition under which any contingent admission control policy is effective over the initial busy period. We also illustrate how this result can be used to derive stability conditions for state-dependent queueing systems with impatient customers.

Section~6 contains the proof of the main result, which relies on a novel coupling construction between the controlled and uncontrolled workload processes.

Finally, in Section~7 we provide a brief conclusion of the current research.

\section{Model description}\label{sec:model}
In this section, we introduce two queueing models that differ in the way arrivals are handled. 
The baseline model admits all customers, whereas the controlled model incorporates an admission control mechanism.

\subsection{Baseline model: no admission control}
\label{subsec:baseline_model}

\paragraph*{Probability space and primitives.}

All random variables are defined on a common probability space 
$(\Omega,\mathcal{F},\mathbb{P})$.

Let $\mathcal{T} \equiv (T_n)_{n=1}^{\infty}$ be a sequence of random variables representing the arrival epochs, satisfying
\[
0 < T_1 < T_2 < \cdots \quad \text{almost surely}.
\]

For each $n\ge1$, let $Y_n:\Omega\to\mathcal{Y}$ be a random variable representing the type of the $n$-th arriving customer, where the type space $\mathcal{Y}$ is an arbitrary topological space. 
In addition, for each $n\ge1$, let $S_n:\Omega\to(0,\infty)$ denote the service requirement of the $n$-th arriving customer in case that the customer joins the queue.

\paragraph*{Workload dynamics.}

Consider a single-server queue operating at a constant unit service rate. 
For an initial workload $x>0$, define the workload process
\[
W^x \equiv (W^x_t)_{t\ge0},
\qquad W^x_0 = x.
\]

The joining decision of the $n$-th customer is represented by a binary random variable 
$I_n \in \{0,1\}$, where $I_n=1$ indicates joining and $I_n=0$ indicates balking.

The workload evolves according to
\begin{equation*}
W_t^x\equiv L_t^x-\inf_{0\leq s<t}\left(L_s^x\right)^-,
\qquad t\ge0,
\end{equation*}
where
\begin{equation*}
L^x_t
\equiv x+\sum_{n:T_n \le t} I_n S_n - t,
\qquad t\ge0,    
\end{equation*}
and $a^-\equiv \min(a,0)$ for every $a\in\mathbb{R}$.

\paragraph*{Statistical assumptions (conditional on $\sigma(\mathcal{T})$).}

All assumptions stated below are made conditionally on $\sigma(\mathcal{T})$. In what follows, assume that $H(\cdot)$ is a probability distribution on $\mathcal{Y}$. In addition, assume that  $\Psi(s;y)$ is a probability kernel, i.e., for each $y\in\mathcal{Y}$, $\Psi(s;y)$ is a distribution function in $s$, and for each $s>0$, $\Psi(s;y)$ is Borel in $y$. In addition, assume that $\Phi:\mathcal{Y}\times[0,\infty)\to\{0,1\}$ is some measurable function. 

\begin{itemize}

\item[\textbf{(A1)}]
The sequence $(Y_n)_{n=1}^{\infty}$ is identically distributed according to $H(\cdot)$, and for each $n\ge1$, the type of the $n$-th arriving customer is independent of the history of the system prior to the arrival epoch $T_n$. That is, $Y_n$ is independent of
\[
\sigma\!\left(
\left\{S_j : 1\le j < n\right\}
\cup
\left\{Y_j : 1\le j < n\right\}\right).
\]

\item[\textbf{(A2)}]
The balking decision of each customer may depend on the customer's type and the workload level at his arrival epoch. Formally, for each $n\ge1$, the decision of the $n$-th arriving customer is provided by the random variable:
\begin{equation*}
    I_n\equiv\Phi\left(Y_n,W^x_{T_n-}\right).
\end{equation*}

\item[\textbf{(A3)}]
The sequence $(S_n)_{n=1}^{\infty}$ is identically distributed, and for each $n\ge1$, the conditional distribution of $S_n$ given 
\begin{equation*}
    \sigma\!\left(
\left\{W^x_t : 0\le t < T_n\right\}
\cup
\left\{Y_j : 1\le j < n\right\}
\right)
\end{equation*}
depends only on $Y_n$ and equals $\Psi(\cdot;Y_n)$.

\end{itemize}

\begin{remark}
\normalfont
Assumption \textbf{(A1)} implies that $(Y_n)_{n=1}^{\infty}$ is an IID sequence.
\end{remark}

\begin{remark}\normalfont Intuitively, Assumption \textbf{(A3)} means that the service requirement of a customer depends only on the customer's type and not on the past evolution of the system.
    
\end{remark}

\begin{remark}\normalfont
Assumptions \textbf{(A1)}--\textbf{(A3)} jointly imply that the service requirements of customers who join the system are mutually independent. 
However, this does not imply independence of the service requirements of all potential customers, including those who eventually balk.
\end{remark}

\begin{remark}\normalfont
The above framework is sufficiently general to accommodate state-dependent and type-dependent behavior. 
In particular, the service requirement $S_n$ may depend on the customer's type $Y_n$, for example when $S_n$ is determined as the solution of an optimization problem parameterized by $Y_n$. 
Such formulations arise naturally in queueing models with discretionary (or endogenous) service requirements; see, e.g., Debo and Li \cite{Debo2021}, Feldman and Segev \cite{Feldman2022}, Gat and Jacobovic \cite{Gat2026}, and Jacobovic \cite{Jacobovic2022a,Jacobovic2022b}.
\end{remark}

\begin{example}\label{remark: example11}\normalfont
A standard framework consistent with the above assumptions is the following. 
Assume that $\mathcal{T}$ and $((S_n,Y_n))_{n=1}^{\infty}$ are independent and that $((S_n,Y_n))_{n=1}^{\infty}$ is an IID sequence such that $Y_1\sim H(\cdot)$ and $S_1$ conditioned on $\sigma(Y_1)$ is distributed according to $\Psi(\cdot;Y_1)$. 
In addition, suppose that each customer decision whether to join or balk is based solely on the customer's type and the workload observed at the arrival epoch. 
Formally, for each $n\ge1$, we assume
\[
I_n \in \sigma\!\left(\{Y_n, W^x_{T_n-}\}\right).
\]    
\end{example}
\subsection{Admission-controlled model}
\label{subsec:admission_control}

We now introduce a model with an admission control mechanism.

\paragraph*{Probability space and primitives.}

All random variables are defined on a common probability space 
$(\hat{\Omega},\hat{\mathcal{F}},\hat{\mathbb{P}})$, which may differ from $(\Omega,\mathcal{F},\mathbb{P})$.

Let $\hat{\mathcal{T}} \equiv (\hat{T}_n)_{n=1}^{\infty}$ be a sequence of random variables representing the arrival epochs, satisfying
\[
0 < \hat{T}_1 < \hat{T}_2 < \cdots \quad \text{almost surely}.
\]

For each $n\ge1$, let $\hat{Y}_n:\hat{\Omega}\to\mathcal{Y}$ denote the type of the $n$-th arriving customer. 
Furthermore, for each $n\ge1$, let $\hat{S}_n:\hat{\Omega}\to(0,\infty)$ denote the service requirement of the $n$-th arriving customer if the customer joins the queue.

Finally, for each $n\ge1$, let $\hat{Z}_n:\hat{\Omega}\to(0,1)$ be a random variable which allows the dispatcher to randomize the admission decision for the $n$-th arriving customer.

\paragraph*{Workload dynamics under control.}

Let $\hat{I}_n\in\{0,1\}$ denote the customer's joining decision and 
$\hat{J}_n\in\{0,1\}$ denote the dispatcher's admission decision.

The controlled workload process
\[
\hat{W}^x \equiv (\hat{W}^x_t)_{t\ge0},
\qquad \hat{W}^x_0 = x,
\]
evolves according to
\begin{equation*}
\hat W_t^x\equiv \hat L_t^x-\inf_{0\leq s<t}\left(\hat L_s^x\right)^-,
\qquad t\ge0,
\end{equation*}
where
\begin{equation*}
\hat L^x_t
=x+\sum_{n:T_n \le t} \hat I_n\hat{J}_n \hat S_n - t,
\qquad t\ge0.    
\end{equation*}

Between arrival epochs\textcolor{red}{,} the workload decreases deterministically at unit rate. 
Thus, the workload process evolves deterministically between successive arrivals and experiences jumps only at arrival epochs of customers who are admitted and willing to join.

\paragraph*{Statistical assumptions.}
The following assumption ensures that the arrival process distribution is the same in both systems. 
\begin{itemize}
    \item[\textbf{($\hat{\text{A}}$0)}] We assume that the distribution of $\hat{\mathcal{T}}$ with respect to $\hat{\mathbb{P}}$ is the same as the distribution of $\mathcal{T}$ with respect to $\mathbb{P}$. 
\end{itemize}
All assumptions stated below are made conditionally on $\sigma(\hat{\mathcal{T}})$.

\begin{itemize}

\item[\textbf{($\hat{\text{A}}$1)}]
The sequence $(\hat{Y}_n)_{n=1}^{\infty}$ is identically distributed according to $H(\cdot)$. Furthermore, for each $n\ge1$, the type of the $n$-th arriving customer is independent of the history of the system prior to his arrival epoch. That is, $\hat{Y}_n$ is independent of
\[
\sigma\!\left(
\left\{\hat{S}_j : 1\le j < n\right\}
\cup
\left\{\hat{Y}_j : 1\le j < n\right\}
\cup
\left\{\hat{Z}_j : 1\le j \leq n\right\}\right).
\]

\item[\textbf{($\hat{\text{A}}$2)}]
The balking decision of each customer may depend on the customer's type and the workload level at his arrival epoch. Formally, for each $n\ge1$, the decision of the $n$-th arriving customer is provided by the random variable:
\begin{equation*}
    \hat{I}_n\equiv\Phi\left(\hat{Y}_n,\hat{W}_{\hat{T}_n-}\right).
\end{equation*}

\item[\textbf{($\hat{\text{A}}$3)}]
For each $n\ge1$, the conditional distribution of $\hat{S}_n$ given
\[
\sigma\!\left(
\left\{\hat{W}^x_t : 0\le t < \hat{T}_n\right\}
\cup
\left\{\hat{Y}_j : 1\le j \leq n\right\}
\cup
\left\{\hat{Z}_j : 1\le j \leq n\right\}\right)
\] 
depends only on $\hat{Y}_n$ and equals $\Psi(\cdot;\hat{Y}_n)$.

\item[\textbf{($\hat{\text{A}}$4)}] At each step the randomization variable is independent from past events. That is, for each $n\geq1$, $\hat{Z}_n$ is uniformly distributed on the unit interval and it is independent from  
\[
\sigma\!\left(
\left\{\hat{W}^x_t : 0\le t < \hat{T}_n\right\}
\cup
\left\{\hat{Y}_j : 1\le j \leq n\right\}
\cup
\left\{\hat{Z}_j : 1\le j < n\right\}\right).
\]

\item[\textbf{($\hat{\text{A}}$5)}]
The dispatcher's decision regarding each customer may depend on the randomization variable and on the history of the workload process until his arrival epoch. Formally, for each $n\ge1$, the random variable $J_n$ indicating the admission of the $n$-th arriving customer is measurable with respect to
\begin{equation*}
    \sigma\!\left(
\left\{\hat W^x_t : 0\le t < T_n\right\}\right)\otimes\sigma\left(\hat{Z}_n\right).
\end{equation*}
\end{itemize}
\begin{remark}
  \normalfont Intuitively, the last phrase in \textbf{($\hat{\text{A}}$2)} ensures that the admission control mechanism does not alter the intrinsic balking behavior of customers. Whenever a customer observes the same information as in the baseline model, the joining decision must coincide with that of the baseline system.  
\end{remark}

\begin{remark}
\normalfont
Note that if $\hat{J}_n \equiv 1$ for all $n\ge1$, i.e., if all arriving customers are admitted, then the present model coincides with the baseline model, which can therefore be viewed as a special case. 

The reason for presenting the two models separately is the following. Defining both models within the framework of the controlled model would implicitly impose a particular coupling between them. In the sequel, we introduce a different, nontrivial coupling of the two models. For clarity of exposition, we therefore prefer the current (admittedly longer) separate presentation.
\end{remark}

\begin{example}\label{remark: example22}\normalfont
A standard framework consistent with the above assumptions can be described as follows. 
Let $\hat{\mathcal{T}}$, $(\hat{Z}_n)_{n\ge 1}$, and $((\hat{S}_n,\hat{Y}_n))_{n\ge 1}$ be independent random objects such that: 
\begin{enumerate}
    \item $\hat{\mathcal{T}}$ is an arrival process  consistent with $\hat{\textbf{A}}$\textbf{0}.

    \item $(\hat{Z}_n)_{n\ge 1}$ is an IID sequence of random variables uniformly distributed on the unit interval.

    \item $((\hat{S}_n,\hat{Y}_n))_{n\ge 1}$ is an IID sequence with $\hat{Y}_1 \sim H(\cdot)$, and, conditionally on $\hat{Y}_1$, the service time $\hat{S}_1$ is distributed according to $\Psi(\cdot;\hat{Y}_1)$.
\end{enumerate}  
Furthermore, assume that each customer decides whether to join or balk based solely on their type and the workload observed at their arrival epoch. Formally, for each $n \ge 1$, we assume
\[
\hat{I}_n \in \sigma\!\bigl(\{\hat{Y}_n, \hat{W}^x_{\hat{T}_n-}\}\bigr).
\]

In addition, assume that for every $n\ge1$, the admission of the $n$-th arriving customer depends only on the randomization variable $\hat{Z}_n$ and the arrival epoch $\hat{T}_n$. That is, conditioned on $\mathcal{T}$ (with $\hat{T}_1,\hat{T}_2,\ldots$ treated as fixed), for every $n\ge1$ we assume
\begin{equation*}
\hat{J}_n \in \sigma(\hat{Z}_n).    
\end{equation*}
Note that once $\hat{J}_n\equiv1$ for every $n\ge1$, then we get back the model discussed in Example~\ref{remark: example11}. 
\end{example}

\section{Effective admission control}\label{sec:effective}

This section formalizes the notion of effectiveness for an admission control policy and clarifies its probabilistic interpretation.

Recall that the baseline and the admission-controlled models give rise to two workload processes, $W^x$ and $\hat{W}^x$, both starting from the same initial level $x \ge 0$. 
By definition, the two systems are distributionally identical, and the difference between them lies solely in the implementation of admission control.

\subsection{Stochastic comparison on path space}

Both workload processes have right-continuous sample paths with left limits and therefore take values in the Skorokhod space $\mathcal{D}([0,\infty))$.

We denote by $\mathbb{P}_{W^x}$ and $\hat{\mathbb{P}}_{\hat{W}^x}$ their respective probability laws on this space. We equip $\mathcal{D}([0,\infty))$ with the natural pointwise partial order: for $u,v \in \mathcal{D}([0,\infty))$, we write
\[
u \preceq v 
\quad \text{if and only if} \quad 
u(t) \le v(t) \text{ for all } t \ge 0.
\]

A measurable functional $f : \mathcal{D}([0,\infty)) \to [0,\infty)$ is said to be nondecreasing if
\[
u \preceq v 
\quad \Longrightarrow \quad 
f(u) \le f(v).
\]

\begin{definition}\label{def:effective}
The admission control policy implemented in the controlled system is said to be \emph{effective} if
\[
\hat{W}^x \preceq_{\mathrm{st}} W^x,
\]
that is, if for every nonnegative nondecreasing measurable functional 
$f : \mathcal{D}([0,\infty)) \to [0,\infty)$,
\[
\hat{\mathbb{E}}f(\hat{W}^x)
\le
\mathbb{E}f(W^x).
\]
Otherwise, the policy is said to be ineffective.
\end{definition}

This definition captures the idea that, under an effective admission control policy, we expect the entire workload trajectory to be stochastically smaller than in the baseline system.

\subsection{The natural coupling}

Since service requirements are strictly positive, both workload processes evolve as piecewise-deterministic processes with upward jumps at arrival epochs of customers who join and negative linear unit drift between the jumps.

If the joining decision of each customer is determined uniquely by his type and independently from the workload history (e.g., when all admitted customers join the system), the two systems can be constructed on a common probability space using identical arrival, types and service sequences. 
Under such a coupling, as illustrated in Figure~1, the controlled system simply removes some upward jumps relative to the baseline system.

In this case, a pathwise ordering holds:
\[
\hat{W}^x(t) \le W^x(t),
\qquad \forall t \ge 0,
\quad \text{almost surely}.
\]
Consequently, stochastic dominance follows immediately.

\begin{figure}[H]
    \centering
    \begin{tikzpicture}[scale=0.9]
        \draw[->] (0,0) -- (11,0) node[right] {$t$};
        \draw[->] (0,0) -- (0,7.1) node[above] {$\textcolor{blue}{W^x_t}, \textcolor{red}{\hat {W}^x_t}$};


        \fill[color = blue] (canvas cs: x = 2cm, y = 0cm) circle(2pt);
        \fill[color = red] (canvas cs: x = 2cm, y = -0.1cm) circle(2pt);
        
        \fill[color = blue] (canvas cs: x = 3cm,y = 0cm) circle (2pt);
        
        \fill[color = blue] (canvas cs: x = 5cm, y = 0cm) circle (2pt);
        \fill[color = red] (canvas cs: x = 5cm, y = -0.1cm) circle(2pt);
        
        \fill[color = blue] (canvas cs: x= 7.5cm, y = 0cm) circle (2pt);
        
        \fill[color = blue] (canvas cs: x = 10cm, y = 0cm) circle (2pt);
        \fill[color = red] (canvas cs: x = 10cm, y = -0.1cm) circle(2pt);

        \draw[line width=0.5mm, color = blue]{} (0cm, 3.1cm) -- (2cm, 1.1cm);
        \draw[line width=0.5mm, color = blue, dashed]{} (2cm, 1.1cm) -- (2cm, 3.6cm);
        \draw[line width=0.5mm, color = blue]{} (2cm, 3.6cm) -- (3cm, 2.6cm);
        \draw[line width=0.5mm, color = blue, dashed]{} (3cm, 2.5cm) -- (3cm, 5cm);
        \draw[line width=0.5mm, color = blue]{} (3cm, 5cm) -- (5cm, 3cm);
        \draw[line width = 0.5mm, color = blue, dashed]{}{} (5cm, 3cm) -- (5cm, 6.5cm);
        \draw[line width=0.5mm, color = blue]{} (5cm, 6.5cm) -- (7.5cm, 4cm);
        \draw[line width=0.5mm, color = blue, dashed]{} (7.5cm, 4cm) -- (7.5cm, 7.5cm);
        \draw[line width=0.5mm, color = blue]{} (7.5cm, 7.5cm) -- (10cm, 5cm);
        \draw[line width=0.5mm, color = blue, dashed]{}{} (10cm, 5cm) -- (10cm, 7cm);
        \draw[line width=0.5mm, color = blue]{}{} (10cm, 7cm) -- (10.5cm, 6.5cm);

        \draw[line width=0.5mm, color = red]{} (0cm, 3cm) -- (2cm, 1cm);
        \draw[line width=0.5mm, color = red, dashed]{} (2cm, 1cm) -- (2cm, 3.5cm);
        \draw[line width=0.5mm, color = red]{} (2cm, 3.5cm) -- (5cm, 0.5cm);
        \draw[line width=0.5mm, color = red, dashed]{} (5cm, 0.5cm) -- (5cm, 4 cm);
        \draw[line width=0.5mm, color = red]{}{} (5cm, 4cm) -- (9cm, 0cm);
        \draw[line width=0.5mm, color = red]{}{} (9cm, 0cm) -- (10cm, 0cm);
        \draw[line width=0.5mm, color = red, dashed]{}{} (10cm, 0cm) -- (10cm, 2cm);
        \draw[line width=0.5mm, color = red]{}{} (10cm, 2cm) -- (10.5cm, 1.5cm);
    \end{tikzpicture} 
    \caption{\setstretch{1.3}
Illustration of the natural coupling between the uncontrolled workload process $W^x$ (blue) and the controlled workload process $\hat{W}^x$ (red) under the assumption that all admitted customers necessarily join. Blue dots on the time axis represent arrival epochs in the uncontrolled system, while red dots correspond to customers admitted to the controlled system. Since every admitted customer necessarily joins the queue and service times are strictly positive, both processes can be constructed on a common probability space so that $\hat{W}^x(t) \le W^x(t)$ for every $t\geq0$. This pathwise ordering immediately yields the stochastic dominance $\hat{W}^x \preceq_{\mathrm{st}} W^x$.}
\end{figure}
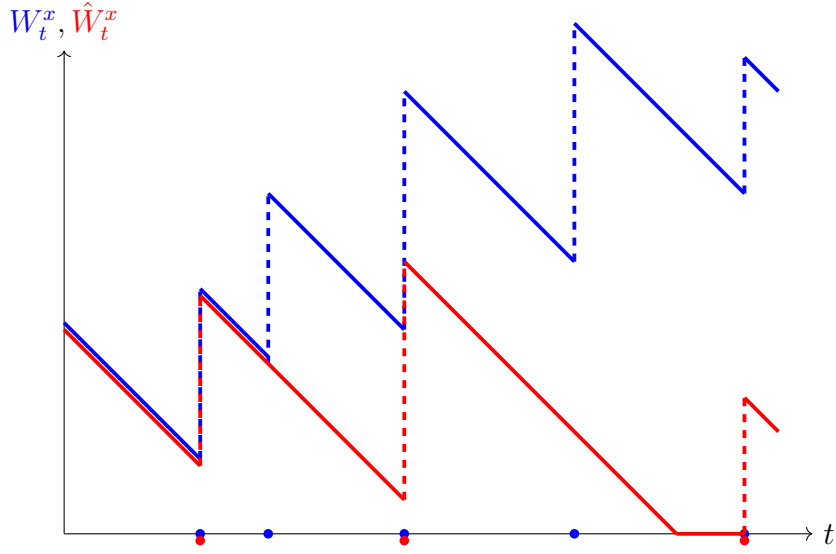

\subsection{The natural coupling under state-dependent joining}
\label{subsec:failure}

The situation changes fundamentally when joining decisions depend on the observed workload.

Suppose that customers are more likely to join when the workload upon arrival is small. 
Consider an arrival time $T_n$ such that
\[
W^x(T_n-) = \hat{W}^x(T_n-).
\]
Assume that the admission policy rejects the $n$-th customer in the controlled system, while in the baseline system the customer joins and generates a positive jump.

Immediately after $T_n$, we then have
\[
\hat{W}^x(T_n) < W^x(T_n).
\]

Now consider the next arrival time $T_{n+1}$. 
If $T_{n+1}$ is close enough to $T_n$, then due to the previous rejection,  
\[
\hat{W}^x(T_{n+1}-) 
<
W^x(T_{n+1}-).
\]
If customers are more inclined to join when the workload is smaller, it may occur that the $(n+1)$-th customer joins in the controlled system but balks in the baseline system.

If the resulting service requirement is sufficiently large, the jump in $\hat{W}^x$ may exceed the level of $W^x$, leading to
\[
\hat{W}^x(T_{n+1})
>
W^x(T_{n+1}).
\]

Thus, as illustrated in Figure~2, the controlled workload may temporarily overtake the baseline workload.

\begin{figure}[H] \centering \begin{tikzpicture}

\draw[->] (0,0) -- (11,0) node[right] {$t$};
\draw[->] (0,0) -- (0,6) node[above] {$\textcolor{blue}{W^x_t}, \textcolor{red}{\hat{W}^x_t}$};

\fill[color = blue] (canvas cs: x = 2cm, y = 0cm) circle(2pt);
\fill[color = red] (canvas cs: x = 2cm, y = -0.1cm) circle(2pt);
\fill[color = blue] (canvas cs: x = 3cm,y = 0cm) circle (2pt);
\fill[color = blue] (canvas cs: x = 5cm, y = 0cm) circle (2pt);
\fill[color = red] (canvas cs: x = 5cm, y = -0.1cm) circle(2pt);
\fill[color = blue] (canvas cs: x= 7.5cm, y = 0cm) circle (2pt);
\fill[color = blue] (canvas cs: x = 10cm, y = 0cm) circle (2pt);
\fill[color = red] (canvas cs: x = 10cm, y = -0.1cm) circle(2pt);

\draw[line width=0.5mm, color = blue] (0cm, 3.1cm) -- (2cm, 1.1cm);
\draw[line width=0.5mm, color = blue, dashed] (2cm, 1.1cm) -- (2cm, 3.6cm);
\draw[line width=0.5mm, color = blue] (2cm, 3.6cm) -- (3cm, 2.6cm);
\draw[line width=0.5mm, color = blue, dashed] (3cm, 2.5cm) -- (3cm, 5cm);
\draw[line width=0.5mm, color = blue] (3cm, 5cm) -- (5cm, 3cm);
\draw[line width=0.5mm, color = blue] (5cm, 3cm) -- (7.5cm, 0.5cm);
\draw[line width=0.5mm, color = blue, dashed] (7.5cm, 0.5cm) -- (7.5cm, 4cm);
\draw[line width=0.5mm, color = blue] (7.5cm, 4cm) -- (10cm, 1.5cm);
\draw[line width=0.5mm, color = blue, dashed] (10cm, 1.5cm) -- (10cm, 3.5cm);
\draw[line width=0.5mm, color = blue] (10cm, 3.5cm) -- (10.5cm, 3cm);

\draw[line width=0.5mm, color = red] (0cm, 3cm) -- (2cm, 1cm);
\draw[line width=0.5mm, color = red, dashed] (2cm, 1cm) -- (2cm, 3.5cm);
\draw[line width=0.5mm, color = red] (2cm, 3.5cm) -- (5cm, 0.5cm);
\draw[line width=0.5mm, color = red, dashed] (5cm, 0.5cm) -- (5cm, 4cm);
\draw[line width=0.5mm, color = red] (5cm, 4cm) -- (9cm, 0cm);
\draw[line width=0.5mm, color = red] (9cm, 0cm) -- (10cm, 0cm);
\draw[line width=0.5mm, color = red, dashed] (10cm, 0cm) -- (10cm, 2cm);
\draw[line width=0.5mm, color = red] (10cm, 2cm) -- (10.5cm, 1.5cm);

\end{tikzpicture} \caption{\setstretch{1.3} Failure of pathwise dominance under state-dependent joining behavior. After an initial rejection in the controlled system, the workload $\hat{W}^x$ becomes strictly smaller than $W^x$, which alters the joining decision of a subsequent arrival. This customer joins the controlled system but balks in the uncontrolled one, generating a jump that causes $\hat{W}^x$ to exceed $W^x$. Blue dots on the time axis represent arrival epochs in the uncontrolled system, while red dots correspond to customers admitted to the controlled system.} \end{figure}
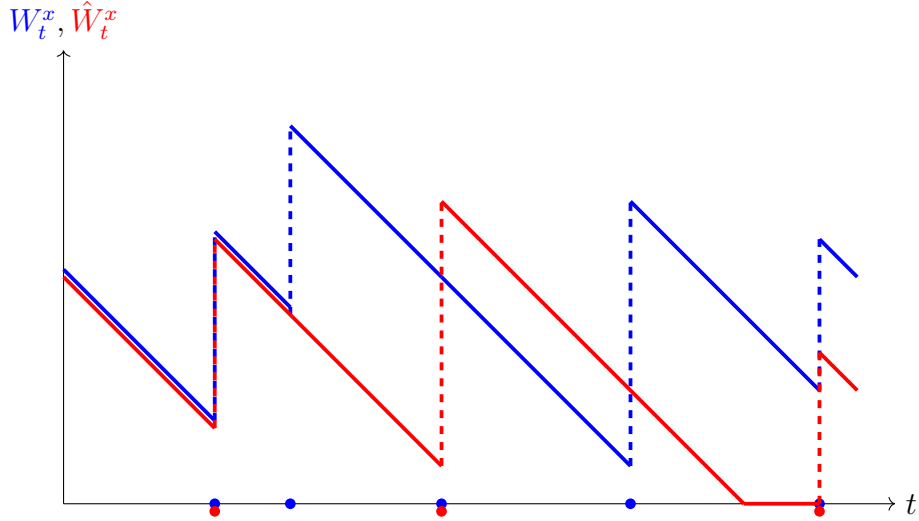

\subsection{Implications}

The above discussion shows that stochastic dominance
\[
\hat{W}^x \preceq_{\mathrm{st}} W^x
\]
cannot be taken for granted in the presence of state-dependent joining decisions.

Although admission control locally reduces workload by removing arrivals, it may indirectly alter future joining behavior in a manner that increases the workload trajectory. 
In the next section, we construct explicit counterexamples demonstrating that stochastic dominance may indeed fail.
\section{Counterexamples}\label{sec:counterexamples}

This section establishes that admission control is not necessarily effective in the sense of Definition~\ref{def:effective}. Namely, our objective is therefore to provide counterexamples that are fully consistent with the probabilistic assumptions introduced earlier and that rigorously demonstrate the possible failure of stochastic dominance. 
We begin by specifying the general probabilistic setting in which the counterexamples are constructed. Note that in all counterexamples, the  two systems are set in their natural coupling defined in $(\Omega,\mathcal{F},\mathbb{P})$.

\subsection{Model setup for the counterexamples}

Both counterexamples are formulated within a single-server queue with impatient customers, defined as follows.

\begin{definition}\label{def:FCFS}
Fix $\lambda_h>0$, $x>0$, a Borel function $\lambda:[0,\infty)\to[0,\lambda_h]$, and a distribution function $\Psi:[0,\infty)\times[0,\infty]\to[0,1]$.
The $\mathrm{M}_t/\mathrm{G}(\Psi)/1+\mathrm{H}(\Psi)$ queue with initial workload $x$ is defined by:
\begin{enumerate}
    \item A single server service system with  infinite waiting room operating under a first-come, first-served (FCFS) discipline.
    \item At time $t=0$, a single customer is present with remaining service time $x$.
    \item Arrivals follow an inhomogeneous Poisson process with rate $\lambda(\cdot)$.
    \item $(S_i,Y_i)_{i\ge1}$ is an IID\ sequence with joint distribution $\Psi$, independent of the arrival process. For each $i\ge1$, the $i$-th customer balks unless 
    \[
        Y_i \ge W^x_{T_i-},
    \]
    in which case the service time equals $S_i$.
\end{enumerate}
\end{definition}

This model fits into the general framework introduced in Section~\ref{sec:model} (recall Example~\ref{remark: example22}). 
In particular, any inhomogeneous Poisson process with rate $\lambda(\cdot)\le\lambda_h$ can be constructed by thinning a homogeneous Poisson process with rate $\lambda_h$ (Lewis and Shedler \cite{Lewis1979}), ensuring consistency with our primitive-model formulation. Specifically:
\begin{itemize}
    \item Let $(\hat{Z}_n)_{n\ge1}$ be an IID\ sequence with $\hat{Z}_1\sim \mathrm{U}[0,1]$.
    \item Let $(\hat{T}_n)_{n\ge1}$ denote the arrival times of a homogeneous Poisson process with rate $\lambda_h$.
    \item Assume independence between $(\hat{Z}_n)_{n\ge1}$ and the Poisson process.
    \item Retain the arrival at time $\hat{T}_n$ if and only if
    \[
        \hat{Z}_n \le \frac{\lambda(\hat{T}_n)}{\lambda_h}.
    \]
\end{itemize}
The resulting process is an inhomogeneous Poisson process with rate $\lambda(\cdot)$.

\begin{remark}
\normalfont
If the arrival process is homogeneous with constant rate $\lambda_h$, the system reduces to the baseline model. In this case, $\lambda(t) = \lambda_h$ for all $t \ge 0$, and the admission policy accepts every arriving customer almost surely.
\end{remark}

\begin{remark}
\normalfont
Under the FCFS discipline, $Y_i$ represents the maximal waiting time customer $i$ is willing to tolerate. Hence, conditional on the arrival epoch, the joining probability is nonincreasing in the observed workload. This monotonicity property will be crucial in the constructions below.
\end{remark}

\subsection{General construction}
In both counterexamples, we consider an $\mathrm{M}_t/\mathrm{G}(\Psi)/1+\mathrm{H}(\Psi)$ queue with initial workload $x=2$ with the natural coupling. 
The uncontrolled system operates at constant rate $\lambda_h$, whereas the controlled system suppresses arrivals on $[0,1)$ and admits arrivals only from time $t=1$ onward:
\[
\lambda(t)=
\begin{cases}
0, & 0\le t<1,\\
\lambda_h, & t\ge1.
\end{cases}
\]

Customer types $Y_1,Y_2,\ldots$ are IID\ with
\[
Y_n=
\begin{cases}
2, & \text{with probability } q,\\
1, & \text{with probability } 1-q,
\end{cases}
\qquad q\in(0,1)\qquad,\ n\geq1,
\]

and service times are deterministic functions of types:
\[
S_n\equiv\varphi(Y_n)\qquad,\ n\geq1.
\]
This specification is consistent with the entire model description provided in Section \ref{sec:model}. 

\subsection{Counterexample 1: large service of impatient customers} The first counterexample is based on the following idea. Suppose that more impatient customers require longer service times. We then construct a system whose uncontrolled version remains sufficiently congested over the interval $[0,1)$, so that only customers with short service requirements are admitted. In contrast, the admission control policy keeps the workload low at time $t=1$, thereby allowing a customer with a large service requirement to enter shortly afterward. Consequently, preventing early arrivals eliminates small contributions but creates the possibility of a substantial jump later on, ultimately worsening expected performance. This mechanism is illustrated in Figure~3(a).
\begin{theorem}[Counterexample 1] \label{thm: counterexample 1}
Assume that 
\[
\varphi(y)=
\begin{cases}
10, & y=1,\\
1,  & y=2.
\end{cases}
\]
Then there exist $q\in(0,1)$ and $\lambda_h\in(0,\infty)$ such that
\[
\hat{W}^2 \not\preceq_{\mathrm{st}} W^2 .
\]
\end{theorem}
\textbf{Proof:} It suffices to show that
\[
\mathbb{E}W_2^2 < \mathbb{E}\hat{W}_2^2 .
\]
For convenience, we split the proof into several steps.\newline\newline\textit{Step 1: Dynamics on $[0,1)$}
.\newline
Since $W_0^2=2$ and the service rate is one, we have
\[
W_t^2 = 2-t \quad \text{as long as no arrival joins.}
\]
Hence $W_t^2>1$ for all $t\in[0,1)$.

Because customer types are supported on $\{1,2\}$ and a customer joins only if 
$Y_i \ge W_{t-}^2$, it follows that during $[0,1)$ only customers with
type $2$ may join the uncontrolled system.

Define an event 
\[
E=\{\text{no customer joins the uncontrolled queue during }[0,1]\}.
\]
On $E$ we have $W_1^2=\hat{W}_1^2=1$.
From time $t=1$ onward both systems evolve according to the same arrival process
and start from the same workload level. Hence,
\[
W_2^2=\hat{W}_2^2 \quad \text{on } E,
\]
and consecutively, it is sufficient to show that 
\begin{equation*}
    \mathbb{E}\left(W_2^2|\bar{E}\right)<\mathbb{E}\left(\hat{W}_2^2|\bar{E}\right).
\end{equation*}
\medskip
\noindent
\textit{Step 2: Pathwise analysis on $\bar{E}$.}\newline
Suppose a Type-2 customer arrives at some time $t<1$.
Immediately before arrival the workload is greater than one, that is, $W_{t-}^2>1$, hence the customer joins and
the workload jumps to
\[
W_t^2 = W_{t-}^2 + 1 > 2 .
\]

After this jump the workload decreases at rate one.
Since $W_t^2>2$ and the time remaining until $1$ is $1-t<1$,
the workload cannot fall below $2$ before time $t=1$.
Consequently no further customer can join before time $t=1$.

Therefore, on $\bar{E}$ exactly one customer joins during $[0,1)$,
and
\[
W_1^2 = 2 ,
\qquad
\hat{W}_1^2 = 1 .
\]

\medskip
\noindent
\textit{Step 3: Upper bound for the uncontrolled system.}\newline
On $\bar{E}$, the uncontrolled system satisfies $W_1^2=2$.
Between times $1$ and $2$, even if a new arrival occurs,
the admission rule implies that the workload cannot exceed $2$
at time $t=2$.
Indeed, if no arrival joins during $[1,2]$, then $W_2^2=1$.
If a Type-2 customer joins, the workload jumps above $2$
but has less than one unit of time to decrease before time $t=2$,
and therefore remains bounded by $2$ at time $t=2$.

Hence,
\[
W_2^2 \le 2 \quad \text{on } \bar{E},
\]
and therefore
\begin{equation}\label{eq:upperbound}
\mathbb{E}\left(W_2^2|\bar{E}\right) \le 2 .
\end{equation}

\medskip
\noindent
\textit{Step 4: Lower bound for the controlled system.}\newline
Let $c$ be the first customer to arrive after $t=1$. On $\bar{E}$ we have $\hat{W}_1^2=1$ and hence if $c$ arrives before $t=2$ he will join the queue (since his type $Y\geq1$). Especially, if he is of Type~1,
then he will join 
and the workload jump by $\varphi(1)=10$.
In that case no additional customer can join before time $t=2$,
and thus
\[
\hat{W}_2^2= 10 .
\]

The probability that $c$ arrives before $t=2$ 
and that he is of Type~1 equals
\[
(1-e^{-\lambda_h})(1-q).
\]

Since $\hat{W}_2^2\ge0$ in all cases,
we obtain
\begin{equation}\label{eq:lowerbound}
\mathbb{E}\left(\hat{W}_2^2|\bar{E}\right)
\ge 10 (1-e^{-\lambda_h})(1-q).
\end{equation}

\medskip
\noindent
\textit{Step 5: Choice of parameters.}

Let $q\downarrow0$ and $\lambda_h\to\infty$.
Then the lower bound in \eqref{eq:lowerbound} converges to $10$,
whereas \eqref{eq:upperbound} yields
\[
\mathbb{E}\left(W_2^2|\bar{E}\right) \le 2.
\]
Hence, for suitable choices of $q$ and $\lambda_h$,
\[
\mathbb{E}\left(W_2^2|\bar{E}\right) < \mathbb{E}\left(\hat{W}_2^2|\bar{E}\right),
\]
which implies
\[
\hat{W}^2 \not\preceq_{\mathrm{st}} W^2 . \ \ \blacksquare
\]
\subsection{Counterexample 2: small but repeated contributions} The second counterexample illustrates a different mechanism. In this case, the uncontrolled system is again initially congested, which prevents most arrivals during the early interval. However, under the admission control policy, the workload is kept low, allowing several smaller customers to enter shortly after the control is lifted. The cumulative effect of these arrivals leads to a higher workload than in the uncontrolled system, demonstrating that even multiple small contributions can reverse the expected ordering. This phenomenon is illustrated in Figure~3(b).
\begin{theorem}[Counterexample 2]
Assume that 
\[
\varphi(y)=
\begin{cases}
0.9, & y=1,\\
1,  & y=2.
\end{cases}
\]
Then there exist $q\in(0,1)$ and $\lambda_h\in(0,\infty)$ such that
\[
\hat{W}^2 \not\preceq_{\mathrm{st}} W^2 .
\]
\end{theorem}
\textbf{Proof:} As in Counterexample~1, it suffices to show
\[
\mathbb{E}(W_2^2 \mid \bar{E})
<
\mathbb{E}(\hat{W}_2^2 \mid \bar{E}),
\]
where on $\bar{E}$ we have $W_1^2=2$ and $\hat{W}_1^2=1$. Note that from now on, all the arguments to follow will refer to sample space realization belonging to $\bar{E}$. For convenience, we split the rest of the proof into several steps. \newline\newline
\textit{Step 1: Uncontrolled system.}\newline
Between times $1$ and $2$ we have
\[
W_t^2 = 2-(t-1)
\]
as long as no arrival joins.

Thus $W_t^2>1$ for all $t<2$.
Hence only Type-2 customers may join.

The Type-2 arrivals form a Poisson process
with rate $\lambda_h q$.

If no such arrival occurs in $[1,2]$, then $W_2^2=1$.
If one such arrival occurs, then the workload jumps by $1$,
and by the same argument as in Counterexample~1,
no further admission is possible before time $t=2$.
In that case $W_2^2=2$.

Therefore, collectively these arguments imply that 
\begin{equation}\label{eq: inequality 11}
\mathbb{E}(W_2^2 \mid \bar{E})
=
e^{-\lambda_h q}\cdot 1
+
(1-e^{-\lambda_h q})\cdot 2
=
2 - e^{-\lambda_h q}.    
\end{equation}
\textit{Step 2: Controlled system.}\newline At time $t=1$ we have $\hat{W}_1^2=1$.

Note that the Type-1 arrival process is Poisson with rate $\lambda_h(1-q)$. Thus, the probability that at the sub-interval $(1,1.1)$ (or $(1.9,2))$, there is at least one arrival such that the first arrival is of type-1, equals
\[
(1-q)\left(1-e^{-0.1\lambda_h}\right).
\]
Since arrivals in  disjoint intervals (e.g., $(1,1.1)$ and $(1.9,2)$) are independent,
the probability that both of them
contain a Type-1 arrival is equal to
\[
(1-q)^2\left(1 - e^{-0.1\lambda_h}\right)^2.
\]

On this event, at least two Type-1 customers join,
each demanding contributing $0.9$ units of service.
Hence, we derive the following lower bound (in fact, it is an equality but for our future purposes, the bound is sufficient)
\[
\hat{W}_2^2 \ge 1.8.
\]

Therefore,
\begin{equation}\label{eq: inequality22}
\mathbb{E}(\hat{W}_2^2 \mid \bar{E})
\ge
1.8(1-q)^2 \left(1 - e^{-0.1\lambda_h}\right)^2.    
\end{equation}
\textit{Step 3: Choice of Parameters.}\newline
Choose $q\equiv q(\lambda_h)\equiv\lambda_h^{-2}$.
Then, $q\to0$ and $\lambda_h q \to 0$
as $\lambda_h\to\infty$.

Hence, with the help of \eqref{eq: inequality 11} and \eqref{eq: inequality22}, deduce that
\[
\lim_{\lambda_h\uparrow\infty}\mathbb{E}(W_2^2 \mid \bar{E}) =1,
\]
while
\[
\liminf_{\lambda_h\uparrow\infty}\mathbb{E}(\hat{W}_2^2 \mid \bar{E})=1.8.
\]
Thus, for sufficiently large $\lambda_h$ (and $q=\lambda_h^2$), we establish the following inequality
\[
\mathbb{E}(W_2^2 \mid \bar{E})
<
\mathbb{E}(\hat{W}_2^2 \mid \bar{E}),
\]
as required. $\blacksquare$\newline
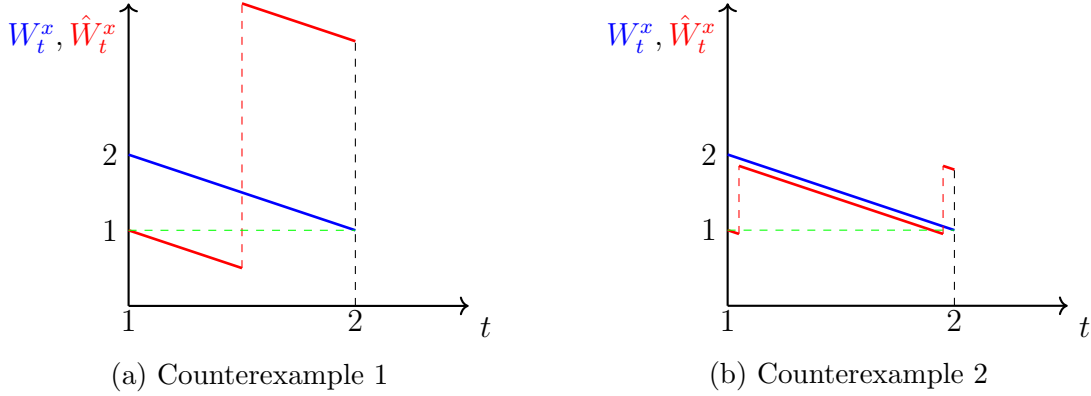
\begin{figure}[H]
\centering
\begin{minipage}{0.48\textwidth}
\centering
\begin{tikzpicture}[x=3cm,y=1cm]  
    \draw[thick,->] (0,0) -- (1.5,0) node[anchor=north west]{$t$};
    \draw[thick,->] (0,0) -- (0,4) node[anchor=north east]{$\textcolor{blue}{W^x_t}, \textcolor{red}{\hat{W}^x_t}$};

    \draw[blue, line width=1pt] (0,2) -- (1,1);
    \node[left] at (0,2) {$2$};

    \draw[red, line width=1pt] (0,1) -- (0.5,0.5);
    \node[left] at (0,1) {$1$};

    \draw[red, dashed] (0.5,0.5)--(0.5,4);
    \draw[red, line width=1pt] (0.5,4)--(1,3.5);

    \draw[dashed] (1,0)--(1,3.5);
    \draw[green,dashed] (0,1)--(1,1);
    \node at (1,-0.2) {$2$};

    \node at (0,-0.2) {$1$};
\end{tikzpicture}

\small (a) Counterexample 1
\end{minipage}
\hfill
\begin{minipage}{0.48\textwidth}
\centering
\begin{tikzpicture}[x=3cm,y=1cm]  
    \draw[thick,->] (0,0) -- (1.5,0) node[anchor=north west]{$t$};
    \draw[thick,->] (0,0) -- (0,4) node[anchor=north east]{$\textcolor{blue}{W^x_t}, \textcolor{red}{\hat{W}^x_t}$};

    \draw[blue, line width=1pt] (0,2) -- (1,1);
    \node[left] at (0,2) {$2$};

    \draw[red, line width=1pt] (0,1) -- (0.05,0.95);
    \node[left] at (0,1) {$1$};

    \draw[red, dashed] (0.05,0.95)--(0.05,1.85);
    \draw[red, line width=1pt] (0.05,1.85)--(0.95,0.95);

    \draw[red, dashed] (0.95,0.95)--(0.95,1.85);
    \draw[red, line width=1pt] (0.95,1.85)--(1,1.8);

    \draw[dashed] (1,0)--(1,1.8);
    \draw[green,dashed] (0,1)--(1,1);
    \node at (1,-0.2) {$2$};

    \node at (0,-0.2) {$1$};
\end{tikzpicture}

\small (b) Counterexample 2
\end{minipage}

\caption{\setstretch{1.3} Illustration of the typical sample path behavior of the uncontrolled (blue) and controlled (red) workload processes on the interval $[1,2]$, on the event $\bar{E}$. 
Panels (a) and (b) depict Counterexamples~1 and~2, respectively. 
For clarity, the dashed green line indicates the level~1 threshold.
}
\label{fig:counterexamples}
\end{figure}
\section{Effectiveness under independence}\label{sec:effectiveness_independence}

Recall Counterexamples~1 and~2 from Section~\ref{sec:counterexamples}. In both cases we showed that
\[
\mathbb{E}W_2^2 < \hat{\mathbb{E}}\hat{W}_2^2.
\]
Since the initial workload in both examples is $x=2$, the time $t=2$ belongs to the initial busy period even if no arrivals occur during $(0,2)$. Consequently, these counterexamples demonstrate that the uncontrolled workload process does not stochastically dominate the controlled workload process over the initial busy period.

Another important feature of both counterexamples is the dependence between the type and the service time of each arriving customer. This naturally raises the following question: can such a counterexample be constructed when the type and the service time are independent?

More precisely, we ask whether it is possible to construct a model satisfying the following two properties:
\begin{enumerate}
    \item[(\textbf{P1})] The uncontrolled workload process fails to stochastically dominate the controlled workload process over the initial busy period.
    \item[(\textbf{P2})] $\Psi(\cdot)\equiv\Psi(\cdot;y)$ is constant in $y\in\mathcal{Y}$, that is, for each arriving customer the type and the service time are independent.
\end{enumerate}

In the remainder of this section we introduce the notion of initial busy-period effectiveness and claim that no model can satisfy both \textbf{P1} and \textbf{P2}. Then, we explain how to apply this result in order to derive stability conditions in state-dependent queues. 

\subsection{Initial busy-period effectiveness}

For each system (controlled and uncontrolled), define the workload process absorbed at zero by considering the stopped processes
\[
V_t^x \equiv W_{t \wedge \tau_x}^x,
\qquad
\hat{V}_t^x \equiv \hat{W}_{t \wedge \hat{\tau}_x}^x,
\qquad t\ge0,
\]
where
\[
\tau_x \equiv \inf\{t\ge0 : W_t^x=0\},
\qquad
\hat{\tau}_x \equiv \inf\{t\ge0 : \hat{W}_t^x=0\}
\]
denote the respective hitting times of $W^x$ and $\hat{W}^x$ at $\{0\}$.

We now introduce the notion of initial busy-period effectiveness.

\begin{definition}\label{def:initial_busy_period_effectiveness}
The admission control policy implemented in $\hat{W}^x$ is said to be \emph{initially busy-period effective} if
\[
\hat{V}^x \preceq_{\mathrm{st}} V^x,
\]
that is, if for every nonnegative, nondecreasing measurable functional 
$f : \mathcal{D}([0,\infty)) \to [0,\infty)$,
\[
\hat{\mathbb{E}}f(\hat{V}^x)
\le
\mathbb{E}f(V^x).
\]
Otherwise, the policy is said to be \emph{initially busy-period ineffective}.
\end{definition}

\begin{remark}\normalfont
Initial busy-period effectiveness, as stated in Definition~\ref{def:initial_busy_period_effectiveness}, is weaker than the notion of effectiveness introduced in Definition~\ref{def:effective}. Observe that in Counterexamples~1 and~2 described in Section~\ref{sec:counterexamples}, the admission control policy is initially busy-period ineffective.
\end{remark}

\begin{remark}\normalfont
Note that the above definition does not require any stability condition ensuring that $\tau_x$ and $\hat{\tau}_x$ have finite mean, or even that they are finite almost surely. This flexibility will be important for the application discussed in Section~\ref{subsec: applications}.
\end{remark}

\subsection{The impossibility result}

We now assume that, for each arriving customer, the type and the service time are independent. Under this assumption we show that the admission control policy is always initially busy-period effective.

The proof relies on an explicit coupling construction under which
\[
\hat{V}^x_t \le V^x_t
\quad \text{a.s. for all } t\ge0 .
\]
The full details of the  proof are provided in Section~\ref{sec: proof}.

\begin{theorem}\label{thm:coupling}
Suppose that \textbf{(P2)} holds. Then \textbf{(P1)} is violated. In other words, if \textbf{(P2)} holds, then the admission control policy implemented in $\hat{W}^x$ is initially busy-period effective, that is,
\[
\hat{V}^x \preceq_{\mathrm{st}} V^x .
\]
\end{theorem}

In the remainder of this section we discuss an application of Theorem~\ref{thm:coupling}.

\subsection{Stability of state-dependent queues}\label{subsec: applications}

We illustrate the usefulness of Theorem~\ref{thm:coupling} by deriving sufficient conditions for the stability of state-dependent queues with impatient customers. The model considered here is a state-dependent version of the $\mathrm{M}/\mathrm{G}/1+\mathrm{H}$ queue.

\begin{definition}\label{def:FCFS}
Fix $\lambda_h>0$ and Borel function $(\ell,w)\mapsto\lambda_{\ell,w}(\cdot)$ such that $\lambda_{\ell,w}\le\lambda_h$ for every $\ell=0,1,2,\ldots$ and $w\ge0$. Let $G(\cdot)$ and $H(\cdot)$ be probability distributions on $(0,\infty)$ and $[0,\infty]$, respectively. The $\mathrm{M}_{\ell,w}/\mathrm{G}/1+\mathrm{H}$ queue is defined as follows:
\begin{enumerate}
    \item A single server service system with infinite waiting room operating under a first-come, first-served (FCFS) discipline.
    
    \item At time $t=0$ the system is empty.
    
    \item Arrivals follow a state-dependent Poisson process with rate $\lambda_{\ell,w}$, where $\ell$ denotes the queue length and $w$ denotes the workload.
    
    \item $(S_i,Y_i)_{i\ge1}$ is an IID\ sequence such that $S_1\sim G(\cdot)$ and $Y_1\sim H(\cdot)$ are independent. For each $i\ge1$, the $i$-th customer balks unless
    \[
    Y_i \ge W_{T_i-},
    \]
    in which case the service time equals $S_i$.
\end{enumerate}
\end{definition}

We are interested in sufficient conditions ensuring that the $\mathrm{M}_{\ell,w}/\mathrm{G}/1+\mathrm{H}$ queue is stable. In particular, note that the workload process regenerates at the end of every busy period, that is, whenever the workload hits the origin from above. Hence it suffices to identify conditions under which the expected length of the first busy period is finite.

Note that the first customer arrives to an empty system and therefore necessarily joins. Conditioning on his service requirement yields a special case of the model described in Section~\ref{sec:model}. The arrival process can be constructed by thinning a homogeneous Poisson process with rate $\lambda_h$ as follows:

\begin{itemize}
\item Let $(\hat{Z}_n)_{n\ge1}$ be an IID\ sequence with $\hat{Z}_1\sim\mathrm{U}[0,1]$.

\item Let $(\hat{T}_n)_{n\ge1}$ denote the arrival times of a homogeneous Poisson process with rate $\lambda_h$.

\item Assume independence between $(\hat{Z}_n)_{n\ge1}$ and the Poisson process.

\item If the arrival at time $\hat{T}_n$ occurs when the queue length is $\ell$ and the workload is $w$, retain it if and only if
\[
\hat{Z}_n \le \frac{\lambda_{\ell,w}}{\lambda_h}.
\]
\end{itemize}

Since the service and patience times are independent, Theorem~\ref{thm:coupling} applies. Consequently, the expected length of the busy period in the $\mathrm{M}_{\ell,w}/\mathrm{G}/1+\mathrm{H}$ system is finite whenever the corresponding busy period in the state-invariant $\mathrm{M}/\mathrm{G}/1+\mathrm{H}$ queue with arrival rate $\lambda_h$ has finite expectation.

It is well known (see, e.g., Baccelli, Boyer and H\'ebuterne \cite{Baccelli1984}) that for the state-invariant system this holds whenever
\begin{equation}\label{eq:sufficient_condition}
\lambda_h\int_0^\infty y\,{\rm d}G(y)\,
\lim_{y\uparrow\infty}\bigl[1-H(y)\bigr]
<1.
\end{equation}

One may also consider a more general framework in which the arrival rate is both state-dependent and time-varying, that is, for every $\ell$ and $w$ there exists an arrival rate function $\lambda_{\ell,w}(\cdot)$ which is bounded from above by $\lambda_h$. If $\lambda_{0,0}(\cdot)$ is periodic, the workload process remains regenerative with regeneration occurring at the first period during which the system is empty.

To show that condition \eqref{eq:sufficient_condition} implies that the expected cycle length is finite, one may apply the same arguments as in the proof of Theorem~3 in Bodas and Jacobovic \cite{Bodas2024}. Additional applications of Theorem~\ref{thm:coupling} can be obtained using the same methodology as in Section~1.3 of \cite{Bodas2024}. To keep the manuscript concise, we omit the details.

\section{Proof of Theorem~\ref{thm:coupling}}\label{sec: proof}

The proof of Theorem \ref{thm:coupling} is divided into two parts. 
First, we provide a coupling construction that places the baseline system 
and the controlled system on a common probability space 
$(\Omega,\mathcal{F},\mathbb{P})$. Since, as to be explained,  the pointwise dominance of 
$V^x$ over $\hat{V}^x$ follows quite immediately from this construction, 
the second part is devoted to verifying that the marginal distributions 
of the workload processes of both systems coincide with those defined in Section~\ref{sec:model}.

\subsection{Coupling construction}

We construct both workload processes on a common probability space $(\Omega,\mathcal{F},\mathbb{P})$. 
On this space, the random variables $\mathcal{T}\equiv(T_n)_{n=1}^\infty$ and  
\begin{equation*}
    \{S_n,Y_n\,;\,n\ge1\}
\end{equation*}
represent the primitives of the uncontrolled system. 

In addition, let $(\tilde{Y}_n)_{n\ge1}$ and $(\hat{Z}_n)_{n\ge1}$ be two sequences such that conditionally on $\mathcal{T}$ are both IID and distributed according to $H(\cdot)$ and the uniform distribution on the unit interval, respectively. 
Conditionally on $\mathcal{T}$, the three sequences
\begin{equation}\label{eq: independence}
   \{S_n,Y_n\,;\,n\ge1\}, \qquad (\tilde{Y}_n)_{n\ge1}, \qquad (\hat{Z}_n)_{n\ge1}
\end{equation}
are assumed to be mutually independent.

We construct $V^x$ consistently with the Assumptions~\textbf{(A1)-(A3)} from 
\begin{equation*}
  \{S_n,Y_n,T_n\,;\,n\ge1\}.  
\end{equation*}
We now construct $\hat{V}^x$ inductively so that
\[
\hat{V}^x_t \le V^x_t 
\quad \text{for all } t\ge0 \quad \mathbb{P}\text{-a.s.}
\]

To begin with, for each $n\ge1$, $\hat{Z}_n$ is the randomization variable used by the dispatcher in the decision regarding the $n$-th arriving customer. 
In addition, assume that the exogenous assumptions~\textbf{($\hat{\text{A}}$2)}, and \textbf{($\hat{\text{A}}$5)} regarding the customers' balking decisions and the dispatcher's admission decisions in the controlled system are satisfied.

For the construction we will need the following notation. 
For each $k \ge 1$, let $n_k$ denote the arrival index of the $k$-th customer who \textit{joins} the uncontrolled queue; that is, $n_k$ is the (random) index such that the customer with arrival number $n_k$ is the $k$-th to join the uncontrolled system.

Note that $\hat{V}^x$ is absorbed at $\{0\}$, i.e., if $\hat{V}^x_t=0$ for some $t>0$, then $\hat{V}^x_s=0$ for every $s\ge t$. 
Thus, to complete the construction of $\hat{V}^x$, we describe the process only up to the first time it hits the set $\{0\}$.

\textit{Step 1: Initial condition.}

At time $0$, both processes start at $x$, and therefore equality holds.

\textit{Step 2: Between arrivals.}

Assume that the arrival processes in both systems are identical, i.e., for every $n\ge1$ we set $\hat{T}_n \equiv T_n$. 
This guarantees Assumption~\textbf{($\hat{\text{A}}$0)}. 
Between arrival epochs both workloads decrease at rate $1$ whenever they are positive. 
Hence the ordering is preserved between jumps.

\textit{Step 3: Arrival epochs.}

Consider an index $n\ge1$ such that $\hat{V}^x_{T_n-}>0$ and $\hat{J}_n=1$.

If $\hat{V}^x_{T_n-}=V^x_{T_n-}$, 
we assign to the arriving customer the same type and service time in both systems, i.e., $\hat{Y}_n\equiv Y_n$ and $\hat{S}_n\equiv S_n$.
Thus, by Assumption~\textbf{(A2)} and Assumption~\textbf{($\hat{\text{A}}$2)}, we have $I_n=\hat{I}_n$. 
Hence the arriving customer either joins or balks in both systems simultaneously, and the ordering is preserved (with equality).

If $\hat{V}^x_{T_n-}<V^x_{T_n-}$ and the number of customers who have already joined the controlled queue  is $k-1\geq0$,
we assign to the arriving customer type $\hat{Y}_n\equiv\tilde{Y}_n$ and service requirement $\hat{S}_n\equiv S_{n_k}$ in the controlled system.

If the customer balks in the controlled system, the inequality is preserved.

Note that $\hat{V}^x_{T_n-}>0$ implies that neither process $V^x$ nor $\hat{V}^x$ has been absorbed at $\{0\}$ before time $T_n$. 
Moreover, $S_{n_k}$ is a service amount that has already been added to $V^x$ (but not to $\hat{V}^x$) prior to time $T_n$. 
Therefore, if the arriving customer joins the controlled system, it follows that $\hat{V}_{T_n}^x \le V_{T_n}^x$ and the ordering is preserved. A graphical illustration of this mechanism is provided in Figure~4.
\begin{figure}[H]
    \centering
    \begin{tikzpicture}[scale=0.8]

\draw[->, thick] (0,0) -- (10,0) node[right] {$t$};
\draw[->, thick] (0,0) -- (0,6.5) node[above] {$\textcolor{blue}{V^x_t}, \textcolor{red}{\hat{V}^x_t}$};

\fill[color = blue] (canvas cs: x = 0cm, y = 0cm) circle(2pt);
\fill[color = blue] (canvas cs: x = 2cm, y = 0cm) circle(2pt);
\fill[color = blue] (canvas cs: x = 5cm, y = 0cm) circle(2pt);
\fill[color = blue] (canvas cs: x = 7cm, y = 0cm) circle(2pt);
\fill[color = red] (canvas cs: x = 5cm, y = -0.1cm) circle(2pt);
\fill[color = red] (canvas cs: x = 7cm, y = -0.1cm) circle(2pt);

\draw[line width=0.5mm, color = blue]{} (0cm, 4.1cm) -- (2cm, 3.1cm);
\draw[line width=0.5mm, color = blue]{} (2cm, 5.1cm) -- (8cm, 2.1cm);
\draw[line width=0.5mm, color = red]{} (0cm, 3.5cm) -- (5cm, 1cm);
\draw[line width=0.5mm, color = red]{} (5cm, 1.6cm) -- (7cm, 0.6cm);
\draw[line width=0.5mm, color = red]{} (7cm, 2.5cm) -- (8cm, 2cm);

\draw[line width=0.5mm, color = green, dashed]{} (0cm, 3.5cm) -- (0cm, 4.1cm);
\draw[line width=0.5mm, color = green, dashed]{} (5cm, 1cm) -- (5cm, 1.6cm);
\draw[line width=0.5mm, color = pink, dashed]{} (2cm, 3.1cm) -- (2cm, 5.1cm);
\draw[line width=0.5mm, color = pink, dashed]{} (7cm, 0.6cm) -- (7cm, 2.6cm);

\fill[color = black] (canvas cs: x = 0cm, y = 3.8cm) circle(3pt);
\fill[color = black] (canvas cs: x = 2cm, y = 4cm) circle(3pt);
\fill[color = black] (canvas cs: x = 5cm, y = 3cm) circle(3pt);
\fill[color = black] (canvas cs: x = 7cm, y = 1.5cm) circle(3pt);
\node at (7,2.1) {$\circ$};
\node at (5,2) {$\circ$};
\node[below] at (0,0) {$t_1$};
\node[below] at (2,0) {$t_2$};
\node[below] at (5,0) {$t_3$};
\node[below] at (7,0) {$t_4$};
\end{tikzpicture}\caption{\setstretch{1.3}
This figure illustrates the coupling construction of $V^x$ and $\hat{V}^x$ used in the proof of Theorem~\ref{thm:coupling}. 
The illustration is given in the setting of an $\mathrm{M}_t/\mathrm{G}(\Psi)/1+\mathrm{H}(\Psi)$ queue as described in Definition~\ref{def:FCFS}, where $\Psi$ has a product form. Assume that up to time $t_1$ the two processes coincide, that is, the sets of customers joining the controlled and uncontrolled systems are identical. 
At time $t_1$, an arrival occurs that is rejected by the dispatcher in the controlled system. 
Since the customer’s type (represented by a black filled dot) exceeds the left limit of the uncontrolled process (depicted by the red curve) at the arrival epoch, the customer joins only the uncontrolled system. A similar event occurs at time $t_2$, where another arriving customer is rejected by the dispatcher and therefore joins only the uncontrolled system. At time $t_3$, an arrival is admitted by the dispatcher. 
Because $\hat{V}^x_{t_3-} < V^x_{t_3-}$, the type of the arriving customer is represented by a black hollow circle in the controlled system (and by a black filled dot in the uncontrolled system). 
In the scenario depicted, the customer joins only the controlled queue. 
His service requirement is then identical to that of the customer who entered the uncontrolled system at time $t_1$. Similarly, the customer arriving at time $t_4$ is admitted and joins only the controlled system. 
His service requirement is matched with that of the customer who entered the uncontrolled system at time $t_2$. A direct calculation shows that $\hat{V}^x_{t_4} = V^x_{t_4}$. This construction precludes the possibility of a single upward jump of $\hat{V}^x$ that would cause it to exceed $V^x$.}
\end{figure}
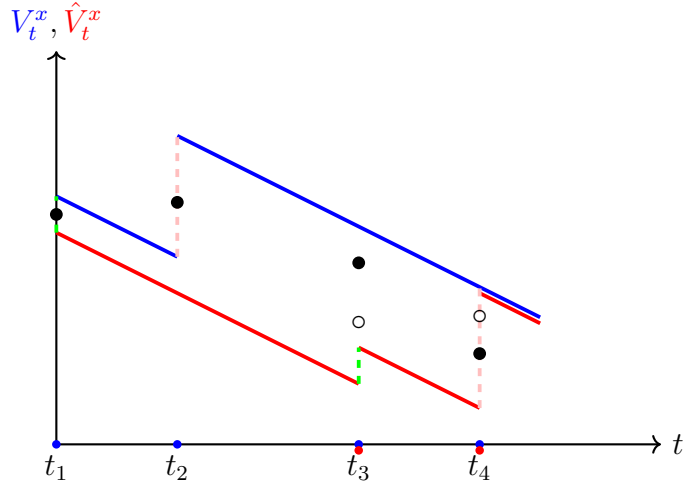
\subsection{Coupling validity}
Note that $\hat{V}^x$ coincides with $\hat{W}^x$ up to the hitting time of $\{0\}$ and from that time it remains zero. Therefore, assuming \textbf{(P2)}, we verify Assumptions \textbf{($\hat{\text{A}}$1)}, \textbf{($\hat{\text{A}}$3)}, and \textbf{($\hat{\text{A}}$4)} on the event that $\hat{W}^x$ has not yet reached the origin. Specifically, for each $n\ge1$, we verify these assumptions on the event $\{\hat{V}_{T_n-}^x > 0\}$, i.e., that $\hat{V}^x$ has not been absorbed before the $n$-th arrival. 

In accordance with the model description, the analysis below is carried out conditionally on $\sigma(\mathcal{T})$. 
Accordingly, we may regard the arrival epochs $T_1,T_2,\ldots$ as a fixed increasing sequence of positive numbers.

We begin by clarifying the terminology used in the coupling construction. 
Define a $k$-th \textit{potential joiner} ($k\geq1$) in the controlled queue as a customer who arrives at that queue after exactly $k-1$ customers have already joined it. Notably, several
potential joiners may arrive at the queue while the number of customers who have already joined the controlled queue remains equal to $k-1$, and therefore several customers may qualify as $k$-th potential joiner under this convention.

Consider now the controlled queue at the arrival epoch $T_n$ of the $n$-th arriving customer. Assume that this customer is a $k$-th potential joiner for some $1\leq k\leq n$. Thus,  by construction, the type of the arriving customer, $\hat{Y}_n$, can be expressed as
\begin{equation}\label{eq: type formula}
\hat{Y}_n
=
\mathbf{1}_{\{M(T_n-)=k-1\}}Y_n
+
\mathbf{1}_{\{M(T_n-)>k-1\}}\tilde{Y}_n,
\end{equation}
where $M(T_n-)$ denotes the number of customers who have joined the \textit{uncontrolled queue} before time $T_n$. 

Similarly, by construction, if the customer arriving at $T_n$ joins the controlled queue, his service time is 
\begin{equation}\label{eq: service formula}
\hat{S}_n
=
\mathbf{1}_{\{M(T_n-)=k-1\}}S_n
+
\mathbf{1}_{\{M(T_n-)>k-1\}}S_{n_k}.
\end{equation}

\textit{Step 1: Verification of \normalfont \textbf{($\hat{\text{A}}$1)}.}

Fix $n\ge1$ and suppose that the $n$-th arriving customer corresponds to a $K$-th potential joiner in the controlled queue for some random $1\le K\leq n$. Let $B$ be a Borel subset of $\mathcal{Y}$ and let $\mathcal{E}_n$ be an event belonging to
\[
\hat{\mathcal{F}}_{T_n-}
\equiv
\sigma\!\left(
\left\{\hat{S}_j : 1\le j < n\right\}
\cup
\left\{\hat{Y}_j : 1\le j < n\right\}
\cup
\left\{\hat{Z}_j : 1\le j \le n\right\}
\right).
\]
 
Using \eqref{eq: type formula}, we obtain
\begin{align*}
\mathbb{P}\!\left(\left\{\hat{V}_{T_n-}^x>0\right\}\cap\hat{Y}_n^{-1}(B)\cap\mathcal{E}_n\right)
&=
\mathbb{P}\!\left(\left\{\hat{V}_{T_n-}^x>0\right\}\cap\{M(T_n-)=K-1\}\cap Y_n^{-1}(B)\cap\mathcal{E}_n\right)\\
&+
\mathbb{P}\!\left(\left\{\hat{V}_{T_n-}^x>0\right\}\cap\{M(T_n-)>K-1\}\cap\tilde{Y}_n^{-1}(B)\cap\mathcal{E}_n\right)\\
&\equiv \Pi_1+\Pi_2 .
\end{align*}
Observe that, by construction, the events
\begin{equation}\label{eq: M-events}
\left\{\hat{V}_{T_n-}^x>0\right\}\cap\left\{M(T_n-)=K-1\right\}
\quad \text{and} \quad
\left\{\hat{V}_{T_n-}^x>0\right\}\cap\left\{M(T_n-)>K-1\right\}
\end{equation}
are determined by the histories of the uncontrolled and controlled systems prior to time $T_n$, that is, by
\[
(W_t)_{0\le t<T_n}\ \ \text{and}\ \ (\hat{W}_t)_{0\le t<T_n}.
\]
Consequently, these events are determined by the random variables 
\[
\{S_j,Y_j,\tilde{Y}_j\,;\,1\le j<n\}
\cup
\{\hat{Z}_j : 1\le j<n\}.
\]
Recalling the definition of $\hat{\mathcal{F}}_{T_n-}$, we deduce that the events
\begin{equation}\label{eq: M-events2}
\mathcal{E}_n\cap\left\{\hat{V}_{T_n-}^x>0\right\}\cap\{M(T_n-)=K-1\}
\quad\text{and}\quad
\mathcal{E}_n\cap\left\{\hat{V}_{T_n-}^x>0\right\}\cap\{M(T_n-)>K-1\}
\end{equation}
are determined by the random variables
\begin{equation}\label{eq: collection1}
\{S_j,Y_j,\tilde{Y}_j\,;\,1\le j<n\}
\cup
\{\hat{Z}_j : 1\le j \le n\}.
\end{equation}

Assumption~\textbf{(A1)} together with the independence stated in \eqref{eq: independence} implies that $(Y_n,\tilde{Y}_n)$ is independent of the random variables belonging to the collection \eqref{eq: collection1}. Hence $(Y_n,\tilde{Y}_n)$ is independent of the events in \eqref{eq: M-events2}, and therefore
\begin{align*}
\Pi_1
&=
\mathbb{P}\!\left[Y_n^{-1}(B)\right]
\mathbb{P}\!\left(\left\{\hat{V}_{T_n-}^x>0\right\}\cap\{M(T_n-)=K-1\}\cap\mathcal{E}_n\right)\\
&=
H(B)\,
\mathbb{P}\!\left(\left\{\hat{V}_{T_n-}^x>0\right\}\cap\{M(T_n-)=K-1\}\cap\mathcal{E}_n\right),
\end{align*}
where the second equality follows from Assumption~\textbf{(A1)}. Similarly,
\begin{align*}
\Pi_2
&=
\mathbb{P}\!\left[\tilde{Y}_n^{-1}(B)\right]
\mathbb{P}\!\left(\left\{\hat{V}_{T_n-}^x>0\right\}\cap\{M(T_n-)>K-1\}\cap\mathcal{E}_n\right)\\
&=
H(B)\,
\mathbb{P}\!\left(\left\{\hat{V}_{T_n-}^x>0\right\}\cap\{M(T_n-)>K-1\}\cap\mathcal{E}_n\right),
\end{align*}
since $\tilde{Y}_n$ is also distributed according to $H(\cdot)$.

Summing the two expressions yields
\[
\mathbb{P}\!\left(\left\{\hat{V}_{T_n-}^x>0\right\}\cap\hat{Y}_n^{-1}(B)\cap\mathcal{E}_n\right)
=\Pi_1+\Pi_2=
H(B)\mathbb{P}\left(\left\{\hat{V}_{T_n-}^x>0\right\}\cap\mathcal{E}_n\right),
\]
which proves \textbf{($\hat{\text{A}}$1)}.

\textit{Step 2: Verification of \normalfont \textbf{($\hat{\text{A}}$3)}.}

Fix $n\ge1$ and suppose that the $n$-th arriving customer corresponds to a $K$-th potential joiner in the controlled queue for some random $1\le K\leq n$. Let $B$ be a Borel subset of $(0,\infty)$ and let $\mathcal{E}_n$ be an event belonging to
\[
\hat{\mathcal{G}}_{T_n-}
\equiv
\sigma\!\left(
\left\{\hat W_t^x;0\le t<T_n\right\}
\cup
\left\{\hat{Y}_j : 1\le j \le n\right\}
\cup
\left\{\hat{Z}_j : 1\le j \le n\right\}
\right).
\]

Using \eqref{eq: service formula}, we write
\begin{align*}
\mathbb{P}\!\left(\left\{\hat{V}_{T_n-}^x>0\right\}\cap\hat{S}_n^{-1}(B)\cap\mathcal{E}_n\right)
&=
\mathbb{P}\!\left(\left\{\hat{V}_{T_n-}^x>0\right\}\cap\{M(T_n-)=K-1\}\cap S_n^{-1}(B)\cap\mathcal{E}_n\right)\\
&\quad+
\mathbb{P}\!\left(\left\{\hat{V}_{T_n-}^x>0\right\}\cap\{M(T_n-)>K-1\}\cap S_{n_{K}}^{-1}(B)\cap\mathcal{E}_n\right)\\
&\equiv P_1+P_2 .
\end{align*}

As in \textit{Step 1}, one may verify that the event
\[
\left\{\hat{V}_{T_n-}^x>0\right\}\cap\{M(T_n-)=K-1\}\cap\mathcal{E}_n
\]
is determined by the random variables 
\[
\{S_j\,;\,1\le j<n\}
\cup
\{Y_j,\tilde{Y}_j,\hat{Z}_j : 1\le j\le n\}.
\]
Therefore, by applying Property \textbf{(P2)} together with Assumption \textbf{(A3)} and the independence stated in \eqref{eq: independence}, we obtain
\begin{align}\label{eq: P1 computation}
P_1
&=
\mathbb{P}\left[S_n^{-1}(B)\right]
\mathbb{P}\!\left(\left\{\hat{V}_{T_n-}^x>0\right\}\cap\{M(T_n-)=K-1\}\cap\mathcal{E}_n\right)\\
&=
\Psi(B)\mathbb{P}\!\left(\left\{\hat{V}_{T_n-}^x>0\right\}\cap\{M(T_n-)=K-1\}\cap\mathcal{E}_n\right).\nonumber
\end{align}

We now turn to the computation of $P_2$. Recall that $\mathcal{E}_n\in\hat{\mathcal{G}}_{T_n-}$ so by construction, for every $1\le k\le \ell\le n$, the event\footnote{Recall that $n_k$ denote the arrival index of the $k$-th customer who joins the uncontrolled queue.} 
\[
\left\{K=k,n_{k}=\ell\right\}
\cap
\left\{\hat{V}_{T_n-}^x>0\right\}
\cap
\{M(T_n-)>k-1\}
\cap
\mathcal{E}_n
\]
is determined by the random variables
\[
\{S_j\,;\,1\le j\leq \ell-1\}
\cup
\{Y_j\,;\,1\le j\leq\ell\}
\cup
\{\tilde{Y}_j,\hat{Z}_j : 1\le j\le n\}.
\]
Hence, using again Property \textbf{(P2)}, Assumption \textbf{(A3)}, and the independence given in \eqref{eq: independence}, we obtain
\begin{align}\label{eq: P2 computation}
P_2
&=
\sum_{1\le k\le\ell\le n}
\mathbb{P}\Big(
\left\{K=k,n_k=\ell\right\}
\cap
\left\{\hat{V}_{T_n-}^x>0\right\}
\cap
\{M(T_n-)>k-1\}
\cap
S_\ell^{-1}(B)
\cap
\mathcal{E}_n
\Big)\nonumber\\
&=
\sum_{1\le k\le\ell\le n}
\mathbb{P}\left[S_\ell^{-1}(B)\right]
\mathbb{P}\Big(
\left\{K=k,n_k=\ell\right\}
\cap
\left\{\hat{V}_{T_n-}^x>0\right\}
\cap
\{M(T_n-)>k-1\}
\cap
\mathcal{E}_n
\Big)\nonumber\\
&=
\Psi(B)
\sum_{1\le k\le\ell\le n}
\mathbb{P}\Big(
\left\{K=k,n_k=\ell\right\}
\cap
\left\{\hat{V}_{T_n-}^x>0\right\}
\cap
\{M(T_n-)>k-1\}
\cap
\mathcal{E}_n
\Big)\nonumber\\
&=
\Psi(B)
\mathbb{P}\left(
\left\{\hat{V}_{T_n-}^x>0\right\}
\cap
\{M(T_n-)>K-1\}
\cap
\mathcal{E}_n
\right).
\end{align}

Combining \eqref{eq: P1 computation} and \eqref{eq: P2 computation}, we obtain
\[
\mathbb{P}\!\left(\left\{\hat{V}_{T_n-}^x>0\right\}\cap\hat{S}_n^{-1}(B)\cap\mathcal{E}_n\right)
=
\Psi(B)
\mathbb{P}\left(\left\{\hat{V}_{T_n-}^x>0\right\}\cap\mathcal{E}_n\right),
\]
from which \textbf{$\hat{\text{A}}3$} follows.

\textit{Step 3: Verification of \normalfont \textbf{($\hat{\text{A}}$4)}.}

Fix $n\ge1$ and on the event $\{\hat{V}_{T_n-}>0\}$, the random variable $\hat{Z}_n$ is used for the first time in the construction of $\hat{V}^x$ at the arrival epoch $T_n$ of the $n$-th arriving customer. 

Moreover, observe that $\hat{Z}_n$ is not used in the construction of $\hat{Y}_n$. Hence, under the proposed construction, the random variables
\begin{equation*}
\left\{\textbf{1}_{\left\{\hat{V}_{T_n-}>0\right\}}\right\}\cup\left\{\hat{W}^x_t : 0\le t < \hat{T}_n\right\}
\cup
\left\{\hat{Y}_k : 1\le k \le n\right\}
\cup
\left\{\hat{Z}_k : 1\le k < n\right\}
\end{equation*}
are measurable with respect to 
\begin{equation*}
\sigma\!\left(
\{S_j,Y_j,T_j : j\ge1\}
\cup
\{\tilde{Y}_k : 1\le k \le n\}
\cup
\{\hat{Z}_k : 1\le k < n\}
\right).
\end{equation*}

Recall that the sequence $(\hat{Z}_n)_{n\ge1}$ consists of IID random variables uniformly distributed on the unit interval. Furthermore, by the independence assumption stated in \eqref{eq: independence}, the variable $\hat{Z}_n$  and independent of
\[
\sigma\!\left(\left\{\textbf{1}_{\left\{\hat{V}_{T_n-}>0\right\}}\right\}\cup
\{S_j,Y_j,T_j : j\ge1\}
\cup
\{\tilde{Y}_k : 1\le k \le n\}
\cup
\{\hat{Z}_k : 1\le k < n\}
\right).
\]

Consequently, $\hat{Z}_n$ is is uniformly distributed on the unit interval and independent of
\[
\sigma\!\left(\left\{\textbf{1}_{\left\{\hat{V}_{T_n-}>0\right\}}\right\}\cup
\{\hat{W}^x_t : 0\le t < \hat{T}_n\}
\cup
\{\hat{Y}_k : 1\le k \le n\}
\cup
\{\hat{Z}_k : 1\le k < n\}
\right),
\]
which verifies Assumption~\textbf{($\hat{\text{A}}$4)}. $\blacksquare$

\section{Conclusion}\label{sec: future research}
This paper revisits the effectiveness of admission control in stochastic service systems with state-dependent customer behavior. While admission control is traditionally expected to reduce congestion by limiting arrivals, our results show that this intuition may fail in the presence of endogenous feedback effects.

Using a stochastic ordering framework, we demonstrate that admission control can, in fact, increase congestion by altering future arrival patterns through the system state. We construct explicit counterexamples illustrating this phenomenon and identify the mechanisms responsible for this unexpected behavior.

On the positive side, we establish a structural condition under which admission control remains effective over the initial busy period. In particular, independence between customers’ types and service requirements eliminates the feedback loop and restores stochastic dominance.

Overall, our results suggest that classical control policies designed under exogenous arrival assumptions may fail to deliver their intended performance when users respond strategically to system conditions. In particular, the performance of admission control policies critically depends on the interaction between control actions and user behavior. This highlights the need for control design in stochastic service systems to explicitly account for endogenous responses to system states.

Our study highlights an important direction for future research: the systematic study of control policies that explicitly account for behavioral feedback, as well as the development of robust mechanisms that remain effective under state-dependent user responses. In particular, it would be of interest to investigate more general forms of feedback between system state and arrivals, to understand the robustness of alternative control policies, and to extend the analysis to networked service systems, where interactions between control and user behavior may be further amplified.


\begin{thebibliography}{99}
\bibitem{Adan1989}
Adan, I., \& Van der Wal, J. (1989). Monotonicity of the throughput of a closed queueing network in the number of jobs. \textit{Operations Research}, \textbf{37}, 953-957.

\bibitem{Ata2006}
Ata, B., \& Shneorson, S. (2006). Dynamic control of an M/M/1 service system with adjustable arrival and service rates. \textit{Management Science}, \textbf{52}, 1778-1791.

\bibitem{Baccelli1984}
Baccelli, F., Boyer, P., \& H\'ebuterne, G. (1984). Single-server queues with impatient customers. \textit{Advances in Applied Probability}, \textbf{16}, 887-905.

\bibitem{Boxma2010}
Boxma, O., Perry, D., Stadje, W., \& Zacks, S. (2010). The busy period of an M/G/1 queue with customer impatience. \textit{Journal of Applied Probability}, \textbf{47}, 130-145.

\bibitem{Boxma2011}
Boxma, O., Perry, D., \& Stadje, W. (2011). The M/G/1+ G queue revisited. \textit{Queueing Systems}, \textbf{67}, 207-220.

\bibitem{Bhattacharya1991}
Bhattacharya, P. P., \& Ephremides, A. (1991). Stochastic monotonicity properties of multiserver queues with impatient customers. \textit{Journal of Applied Probability}, \textbf{28}, 673-682.

\bibitem{Bodas2024}
Bodas, S. A., \& Jacobovic, R. (2024). Useful stochastic bounds in time-varying queues with service and patience times having general joint distribution. \textit{arXiv preprint arXiv:2406.12745}.

\bibitem{Carr2000}
Carr, S., \& Duenyas, I. (2000). Optimal admission control and sequencing in a make-to-stock/make-to-order production system. \textit{Operations research}, \textbf{48}, 709-720.

\bibitem{Cohen2024}
Cohen, A., Subramanian, V., \& Zhang, Y. (2024). Learning-based optimal admission control in a single-server queuing system. \textit{Stochastic systems}, \textbf{14}, 69-107.

\bibitem{Debo2021}
Debo, L., \& Li, C. (2021). Design and pricing of discretionary service lines. Management Science, \textbf{67}, 2251-2271.

\bibitem{Feldman2022}
Feldman, P., \& Segev, E. (2022). The important role of time limits when consumers choose their time in service. \textit{Management Science}, \textbf{68}, 6666-6686.

\bibitem{Gat2026}
Gat, Y., \& Jacobovic, R. (2026). Priority queues with discretionary services. \textit{Unpublished manuscript}.

\bibitem{Hassin2016}
Hassin, R. (2016). \textit{Rational queueing}. CRC press.

\bibitem{Hassin2003}
Hassin, R., \& Haviv, M. (2003). \textit{To queue or not to queue: Equilibrium behavior in queueing systems (Vol. 59)}. Springer Science \& Business Media.

\bibitem{Haviv1998}
Haviv, M., \& Puterman, M. L. (1998). Bias optimality in controlled queueing systems. \textit{Journal of Applied Probability}, \textbf{35}, 136-150.

\bibitem{Heyman1982}
Heyman, D. P. (1982). On Ross's conjectures about queues with non-stationary Poisson arrivals. \textit{Journal of Applied Probability}, \textbf{19}, 245-249.

\bibitem{Jacobovic2022a}
Jacobovic, R. (2022). Internalization of externalities in queues with discretionary services. \textit{Queueing Systems}, \textbf{100}, 453-455.

\bibitem{Jacobovic2022b}
Jacobovic, R. (2022). Regulation of a single-server queue with customers who dynamically choose their service durations. \textit{Queueing Systems}, \textbf{101}, 245-290.


\bibitem{Jouini2007}
Jouini, O., \& Dallery, Y. (2007). Monotonicity properties for multiserver queues with reneging and finite waiting lines. \textit{Probability in the Engineering and Informational Sciences}, \textbf{21}, 335-360.

\bibitem{Knudsen1972}
Knudsen, N. C. (1972). Individual and social optimization in a multiserver queue with a general cost-benefit structure. \textit{Econometrica: Journal of the Econometric Society}, 515-528.

\bibitem{Lewis1979}
Lewis, P. W., \& Shedler, G. S. (1979). Simulation of nonhomogeneous Poisson processes by thinning. \textit{Naval research logistics quarterly}, \textbf{26}, 403-413.

\bibitem{Maglaras2005}
Maglaras, C., \& Zeevi, A. (2005). Pricing and design of differentiated services: Approximate analysis and structural insights. \textit{Operations Research}, \textbf{53}, 242-262.

\bibitem{Moyal2022}
Moyal, P., \& Perry, O. (2022). Many-server limits for service systems with dependent service and patience times. \textit{Queueing Systems}, \textbf{100}, 337-339.

\bibitem{Naor1969}
Naor, P. (1969). The regulation of queue size by levying tolls. \textit{Econometrica: journal of the Econometric Society}, 15-24.

\bibitem{Rolski1981}
Rolski, T. (1981). Queues with non-stationary input stream: Ross's conjecture. \textit{Advances in Applied Probability}, \textbf{13}, 603-618.

\bibitem{Rolski1986}
Rolski, T. (1986). Upper bounds for single server queues with doubly stochastic Poisson arrivals. \textit{Mathematics of Operations Research}, \textbf{11}, 442-450.


\bibitem{Ross1978}
Ross, S. M. (1978). Average delay in queues with non-stationary Poisson arrivals. \textit{Journal of Applied Probability}, 15, 602-609.


\bibitem{Shanthikumar1986}
Shanthikumar, J. G., \& Yao, D. D. (1986). The effect of increasing service rates in a closed queuing network. \textit{Journal of Applied Probability}, \textbf{23}, 474-483.

\bibitem{Shanthikumar1987}
Shanthikumar, J. G., \& Yao, D. D. (1987). Stochastic monotonicity of the queue lengths in closed queueing networks. \textit{Operations Research}, \textbf{35}, 583-588.

\bibitem{Shanthikumar1989}
Shanthikumar, J. G., \& Yao, D. D. (1989). Stochastic monotonicity in general queueing networks. \textit{Journal of Applied Probability}, \textbf{26}, 413-417.

\bibitem{Van Dijk1989}
Van Dijk, N. M., \& van der Wal, J. (1989). Simple bounds and monotonicity results for finite multi-server exponential tandem queues. \textit{Queueing Systems}, \textbf{4}, 1-15.

\bibitem{Van Nunen1983}
Van Nunen, J. A. E. E., \& Puterman, M. L. (1983). Computing optimal control limits for GI/M/s queuing systems with controlled arrivals. \textit{Management Science}, \textbf{29}, 725-734.

\bibitem{Ward2008}
Ward, A. R., \& Kumar, S. (2008). Asymptotically optimal admission control of a queue with impatient customers. \textit{Mathematics of Operations Research}, \textbf{33}, 167-202.

\bibitem{Wu2019}
Wu, C., Bassamboo, A., \& Perry, O. (2019). Service system with dependent service and patience times. \textit{Management Science}, \textbf{65}, 1151-1172.

\bibitem{Yu2023}
Yu, L., \& Perry, O. (2023). Many-server heavy-traffic limits for queueing systems with perfectly correlated service and patience times. \textit{Mathematics of Operations Research}, \textbf{48}, 1119-1157.
\end{thebibliography}
\end{document}